\documentclass[review]{elsarticle}

\usepackage{amsmath,amsthm}
\usepackage{latexsym}
\usepackage{amssymb}
\usepackage{amscd}
\usepackage{amsfonts}
\usepackage[all]{xy}
\usepackage[toc,page]{appendix}
\usepackage{subfig}
\usepackage{url}

\usepackage{tabularx}
\newcolumntype{Y}{>{\small\raggedright\arraybackslash}X}

\usepackage{lineno,hyperref}
\modulolinenumbers[5]

\newtheorem{theorem}{Theorem}[section]

\theoremstyle{definition}

\newtheorem{thm}{Theorem}[section]

\newtheorem{defi}[thm]{Definition}
\newtheorem{prop}[thm]{Proposition}

\newtheorem{lem}[thm]{Lemma}
\newtheorem{rmk}[thm]{Remark}

\newtheorem*{notations}{Notations}
\newtheorem*{acknowledgement}{Acknowledgement}

\theoremstyle{remark}

\def\ord{\textup{ord}}
\def\Fix{\textup{Fix}}
\def\Pic{\textup{Pic}}

\makeatletter
\def\ps@pprintTitle{%
 \let\@oddhead\@empty
 \let\@evenhead\@empty
 \def\@oddfoot{}%
 \let\@evenfoot\@oddfoot}
\makeatother

\begin{document}

\begin{frontmatter}

\title{Quasiphantom categories on a family of \\ surfaces isogenous to a higher product}

\author[mymainaddress]{Hyun Kyu Kim}\corref{mycorrespondingauthor}
\cortext[mycorrespondingauthor]{Corresponding author}
\ead{hkim@kias.re.kr, hyunkyu87@gmail.com}

\author[mysecondaryaddress]{Yun-Hwan Kim}
\ead{yunttang@snu.ac.kr}

\author[mythirdaddress]{Kyoung-Seog Lee}
\ead{kyoungseog02@gmail.com}

\address[mymainaddress]{School of Mathematics, Korea Institute for Advanced Study, 
85 Hoegiro, Dongdaemun-gu, 
Seoul 02455, Republic of Korea}
\address[mysecondaryaddress]{Department of Mathematical Sciences, Seoul National University, GwanAkRo 1, Gwanak-Gu, Seoul 08826, Republic of Korea}
\address[mythirdaddress]{Center for Geometry and Physics, Institute for Basic Science (IBS), Pohang 37673, Republic of Korea}

\begin{abstract}
We construct exceptional collections of line bundles of maximal length 4 on $S=(C \times D)/G$ which is a surface isogenous to a higher product with $p_g=q=0$ where $G=G(32,27)$ is a finite group of order 32 having number 27 in the list of Magma library. From these exceptional collections, we obtain new examples of quasiphantom categories as their orthogonal complements.
\end{abstract}


\end{frontmatter}

\section{Introduction}
Quasiphantom categories are surprising new subcategories in the derived categories of algebraic varieties first discovered by B\"{o}hning, Bothmer and Sosna in \cite{BBS}. Their discovery provides new perspectives on the study of derived categories of algebraic varieties and recently many examples of quasiphantom categories were constructed by many authors(see \cite{AO, BBKS, BBS, CL, Coughlan, Fakhruddin, GKMS, GO, GS, Keum, Lee, Lee2, LS} for more details). However their structures are quite mysterious and we do not know whether every surface of general type with $p_g=q=0$ has a quasiphantom category in its derived category. 

A surface $S$ which is isomorphic to $(C \times D)/G$ where $C, D$ are curves of genus $\geq 2$ and $G$ is a finite group acting on $C \times D$ freely is called a surface isogenous to a higher product. Surfaces isogenous to a higher product are interesting and important classes of surfaces of general type. They play an important role in the study of moduli spaces of general type surfaces. Bauer, Catanese and Grunewald classified surfaces isogenous to a higher product of unmixed type with $p_g=q=0$ in \cite{BCG2}. In particular, they proved that the possible list of groups $G$ is $\mathbb{Z}_2^3$, $\mathbb{Z}_2^4$, $\mathbb{Z}_3^2$, $\mathbb{Z}_5^2$, $\mathbb{Z}_2 \times D_4$, $S_4$, $\mathbb{Z}_2 \times S_4$, $G(16,3)$, $G(32,27)$, $A_5$.

In \cite{Coughlan, GS, Lee, Lee2, LS}, the authors constructed quasiphantom categories in derived categories of surfaces isogenous to a higher product except $G=G(32,27)$, $A_5$ cases. It is very natural to expect that derived categories of all surfaces isogenous to a higher product with $p_g=q=0$ have quasiphantom categories in their derived categories. However quasiphantom categories in derived categories of surfaces had not been constructed for $G=G(32,27)$, $A_5$ cases. 

In this paper we construct exceptional collections of line bundles of maximal length 4 on $S=(C \times D)/G$ which is a surface isogenous to a higher product with $p_g=q=0$ and $G$ is $G(32,27)$. From these exceptional collections we can obtain new examples of quasiphantom categories. They are obtained as the orthogonal complements of the exceptional collections of maximal length 4 on $S$.

\begin{acknowledgement}
We are grateful to Changho Keem, Young-Hoon Kiem for their words of encouragement. We thank Fabrizio Catanese for answering many questions and helpful discussions. We also thank Alexander Kuznetsov for his suggestion to use a computer program to compute the derived categories of the family of surfaces considered in this paper. Part of this work was done when the third named author was a research fellow of Korea Institute for Advanced Study. He thanks KIAS for wonderful working condition and kind hospitality. This work was also supported by IBS-R003-Y1. Last but not least, we thank the referees for their valuable comments which helped us to improve the manuscript.
\end{acknowledgement}

\begin{notations}
We will work over $\mathbb{C}$. Derived category of a variety will mean the bounded derived category of coherent sheaves on the variety. In this paper, $G$ denotes a finite group, $\widehat{G}=Hom(G,\mathbb{C}^*)$ denotes the character group of $G$, and $G(32,27)$ means the finite group of order 32 having number 27 in the list in Magma library.
\end{notations}

\section{Preliminaries}
In this section we recall several definitions and facts which we will use later.
\subsection{Surfaces isogenous to a higher product with $p_g=q=0$}
We review the basic theory of surfaces isogenous to a higher product with $p_g=q=0$.
\begin{defi}\cite[Definition 2.1]{BCF}
A surface $S$ of general type is said to be isogenous to a higher product if $S$ is isomorphic to $(C \times D)/G,$ where $C$ and $D$ are curves of genus at least 2, and $G$ is a finite group acting freely on $C \times D$. We call $S$ of unmixed type if $G$ acts diagonally on $C \times D$.
\end{defi}

Bauer, Catanese and Grunewald classified surfaces isogenous to a higher product with $p_g=q=0$ in \cite{BCG2}. Their strategy was to classify all groups acting freely on product of two curves with some specified conditions. Let $ S = (C \times D)/G $ be a surface isogenous to a higher product of unmixed type with $p_g=q=0$. Consider the two quotient maps $C \to C/G$ and $D \to D/G$. Since $q=0$, we see that $C/G \cong \mathbb{P}^1 \cong D/G$. These quotient maps give several group theoretic data. We recall several terminologies following \cite{BCG2}.


\begin{defi}\cite{BCG2}
Let $G$ be a group and $r$ be a natural number with $r \geq 2$.
	\begin{enumerate}[(1)]
		\item An $r$-tuple $T=[g_1,\cdots,g_r] \in G^r$ is called a spherical system of generators of $G$ if $g_1,\cdots,g_r$ is a system of generators of $G$ and $g_1 \cdots g_r=1$.
		\item Let $A=[m_1, \cdots, m_r]$ be an $r$-tuple of natural numbers with $2 \leq m_1 \leq \cdots \leq m_r$ then a spherical system of generators $T=[g_1,\cdots,g_r]$ is said to have type $A=[m_1, \cdots, m_r]$ if there is a permutation $\tau \in S_r$ such that $\ord(g_i)=m_{\tau(i)}$ for all $i$. 
		\item The stabilizer set $\Sigma(T)$ of a spherical system of generators $T=[g_1,\cdots,g_r]$ is defined as follows: 
$$ \Sigma(T)= \bigcup_{g \in G} \bigcup_{j \in \mathbb{Z}} \bigcup_{i=1}^r \{g g_i^j g^{-1}\} =\{gg^j_ig^{-1}\mid g\in G, j\in \mathbb{Z}, i=1,\cdots, r \}.$$
		\item Let $T_1,T_2$ be a pair of spherical systems of generators $(T_1,T_2)$ of $G$ is called disjoint if $\Sigma(T_1) \cap \Sigma(T_2)=\{ 1 \}$.
	\end{enumerate} 
\end{defi}

\begin{defi}\cite[Definition 1.1]{BCG2}
An unmixed ramification structure of type $(A_1,A_2)$ for $G$ is a disjoint pair $(T_1,T_2)$ of spherical systems of generators of $G,$ such that $T_1$ has type $A_1$ and $T_2$ has type $A_2$. We define $\mathcal{B}(G;A_1,A_2)$ to be the set of unmixed ramification structures of type $(A_1,A_2)$ for $G$.
\end{defi}

Whenever we have a surface $S=(C \times D)/G$ we have an unmixed ramification structure of type $(A_1,A_2)$. Moreover we see that an unmixed ramification structure gives a surface isogenous to a higher product by the following proposition. See \cite{BCG2} for more details.

\begin{prop}\cite[Proposition 2.5]{BCG2}
Let $G$ be a finite group and $A_1, A_2$ be two tuples of natural numbers. Then for any ramification structure $\mathcal{T} \in \mathcal{B}(G;A_1,A_2),$ there is a surface isogenous to a higher product of unmixed type with $G(S)=G$ and $\mathcal{T}(S)=\mathcal{T}$.
\end{prop}

\subsection{Automorphisms of curves}
We review several results about automorphisms of curves and their invariants. First, let us recall Lefschetz fixed point formula. 

\begin{thm}(Lefschetz Fixed Point Formula)\cite[Corollary 12.3]{Breuer}
Let $g \in G$ be a non-trivial automorphism of an algebraic curve $C$ of genus $g \geq 2,$ and let $\chi_{K_C}$ the character of the action of $G$ on $H^0(C,K_C)$. Then we have
$$ \chi_{K_C}(g) + \overline{\chi_{K_C}(g)} = 2 - |\Fix(g)|. $$
In particular, $\chi_{K_C}(g)=1-\frac{1}{2}|\Fix(g)|$ when all characters are real-valued.
\end{thm}

Beauville studied theta characteristics on curves with involution in \cite{Beauville}. His results are the main tools of our investigation of involution invariant line bundles. Let $C$ be a curve and $\sigma$ be the involution on $C$. Let $B$ be the quotient curve $C/ \langle \sigma \rangle,$ $\pi \colon C \to B$ be the quotient map and $R \subset C$ be the set of ramification points. This double covering corresponds to a line bundle $\rho$ on $B$ such that $\rho^2=\mathcal{O}_B(\pi_*R)$. See \cite{Beauville} for more details. Beauville obtained the following result which tells us which $\sigma$-invariant line bundles on $C$ comes from $B$.

\begin{lem} \cite{Beauville} \label{Beauville's lemma1}
Consider the map $\phi \colon \mathbb{Z}^R \to \Pic(C)$ which maps $r \in R$ to $\mathcal{O}_C(r)$. Its image lies in the subgroup $\Pic(C)^\sigma$ of $\sigma$-invariant line bundles. We get a short exact sequence $$ 0 \to \mathbb{Z}/2 \to (\mathbb{Z}/2)^R \to \Pic(C)^{\sigma}/\pi^*\Pic(B) \to 0, $$
and the kernel is generated by $(1,\cdots,1)$.
\end{lem}

And he obtains the following result which tells us how to compute the sheaf cohomology groups of invariant theta characteristics as follows.

\begin{prop} \cite{Beauville} \label{Beauville's lemma2}
Let $\kappa$ be a $\sigma$-invariant theta characteristic on $C$. Then
	\begin{enumerate}[(1)]
		\item $\kappa \cong \pi^*L(E)$ for some $L \in \Pic(B)$ and $E \subset R$ with $L^2 \cong K_B \otimes \rho(-\pi_*E)$. If another pair $(L',E')$ satisfies $\kappa \cong \pi^*L'(E'),$ we have $(L',E')=(L,E)$ or $(L',E')=(K_B \otimes L^{-1},R-E)$. 
		\item $h^0(\kappa) = h^0(L) + h^1(L),$ and the parity of $\kappa$ is equal to $\deg(L)-b+1~({\rm mod}~2)$ where $b$ is the genus of $B$.	
	\end{enumerate}
\end{prop}

Finally we recall the following result of Dolgachev which enables us to construct $G$-equivariant line bundles on $C \times D$ using $G$-invariant line bundles on $C$ and $D$.

\begin{prop} \cite{D2} \label{Dolgachev's proposition}
Let $X$ be a smooth projective variety and let $G$ be a finite group acting on $X$. There is a well-known exact sequence
\[ 0 \to \widehat{G} \to \Pic^G(X) \to \Pic(X)^G \to H^2(G,\mathbb{C}^*), \]
and the last homomorphism is surjective when $X$ is a curve.
\end{prop}
\section{Derived categories of surfaces isogenous to a higher product with $G=G(32,27)$}
Let $S=(C \times D)/G$ be a surface isogenous to a higher product with $p_g=q=0$. It is easy to see that the maximal possible length of an exceptional collection is less than or equal to 4(see \cite{Coughlan, GS, Lee, Lee2, LS} for more details). In this section we construct exceptional collections of line bundles of maximal length 4 on $S=(C \times D)/G$ where $G=G(32,27)$.  The method of construction is similar to other cases. We can construct $G$-invariant ineffective theta characteristics on $C$. However there is no $G$-equivariant ineffective line bundle of degree 8 on $D$. In order to overcome this situation we need to show that there are enough characters of $G$ to construct exceptional collection of length 4. To do this we use computer algebra system Magma \cite{Magma} and group theoretic properties of $G$. See \ref{appendix A} and \ref{appendix B} for more details. In particular, see \ref{appendix A} for the character table of $G$.

\subsection{Equivariant geometry of $C$ and $D$}
Let $g_1, g_2, \cdots, g_5$ be generators of $G$ for the presentation of $G$ as in Appendix A. The unmixed ramification structure $\mathcal{T}=(T_1,T_2)$ corresponds to $S=(C \times D)/G$ is described in \cite{BCG2}. They computed that $\mathcal{T}=(T_1,T_2)$ is of type $([2,2,2,4],[2,2,4,4])$ on $G(32,27)$ which is equivalent to
\[
T_1=[g_1g_4g_5,g_2g_3g_4g_5,g_2g_4g_5,g_1g_3g_4],\;\;
T_2=[g_2g_3g_4,g_2,g_1g_2g_3g_5,g_1g_2]
\]
(\cite{BCG2}, see also \cite{BCG} and Appendix A).
We are going to construct the desired line bundles from the above unmixed ramification structure. We can compute the representation of $H^0(C,K_C)$ by the Lefschetz fixed point formula. 
For $T_1,$ the numbers of fixed points are given by Table \ref{fixedC}. (Note that $1$ fixes every point in $C$.)

\begin{table}[!ht]
\begin{center}
{\setlength\tabcolsep{1,1pt} 
\begin{tabular}{|c||c|c|c|c|c|c|c|c|c|c|c|c|c|c|} \hline
conjugacy class & $1$ & $g_5$ & $g_4$ & $g_4 g_5$ & $g_2 g_3 g_5$ & $g_2$ & $g_2 g_3$ & $g_3 g_4$ & $g_2 g_5$ & $g_3$  & $g_1$ & $g_1 g_2 g_3$ & $g_1 g_2$ & $g_1 g_3$ \\ \hline  
\#fixed points & $\infty$ & 8 & 0 & 0 & 0 & 0 & 8 & 0 & 8 & 0 & 4 & 0 & 0 & 4    \\ \hline
\end{tabular}}
\caption{The number of fixed points for ${T}_1$}
\label{fixedC}
\end{center}
\end{table}

Therefore we get the character $\chi_{K_C}$ of the action of $G$ on $H^0(C,K_C)$. The value of $\chi_{K_C}$ at the identity class is the genus $5$. At any non-trivial conjugacy class of $g \in G(32,27),$ the value of $\chi_{K_C}$ at $g$ is given by 
$$ \chi_{K_C}(g) = \frac{1}{2}(2 - |\Fix(g)|) $$
since every character of $G(32,27)$ is a real-valued function; See Remark \ref{allcharactersarereal}.
Thus the values of the character $\chi_{K_C}$ at the fourteen conjugacy classes ordered as above are as following Table \ref{charactertableC}.

\begin{table}[!ht]
\begin{center}
{\setlength\tabcolsep{2,1pt}  
\begin{tabular}{|c||c|c|c|c|c|c|c|c|c|c|c|c|c|c|} \hline
conjugacy class & $1$ & $g_5$ & $g_4$ & $g_4 g_5$ & $g_2 g_3 g_5$ & $g_2$ & $g_2 g_3$ & $g_3 g_4$ & $g_2 g_5$ & $g_3$  & $g_1$ & $g_1 g_2 g_3$ & $g_1 g_2$ & $g_1 g_3$ \\ \hline 
$ \chi_{K_C} $ & 5 & $-3$ & 1 & 1 & 1 & 1 & $-3$ & 1 & $-3$ & 1 & $-1$ & 1 & 1 & $-1$ \\ \hline
\end{tabular}}
\caption{The values of $\chi_{K_C}$}
\label{charactertableC}
\end{center}
\end{table}

Now from the character table of $G(32,27)$(see Appendix A), we can see that 
$\chi_{K_C} = \chi_7 + \chi_9 + \chi_{11}$.
Similarly, for $T_2$, the numbers of fixed points are given by Table \ref{fixedD}.

\begin{table}[!th]
\begin{center}
{\setlength\tabcolsep{1,1pt} \begin{tabular}{|c||c|c|c|c|c|c|c|c|c|c|c|c|c|c|} \hline
conjugacy class & $1$ & $g_5$ & $g_4$ & $g_4 g_5$ & $g_2 g_3 g_5$ & $g_2$ & $g_2 g_3$ & $g_3 g_4$ & $g_2 g_5$ & $g_3$  & $g_1$ & $g_1 g_2 g_3$ & $g_1 g_2$ & $g_1 g_3$ \\ \hline 
\#fixed points & $\infty$  & 0 & 8 & 8 & 8 & 8 & 0 & 0 & 0 & 0 & 0 & 4 & 4 & 0    \\ \hline
\end{tabular}}
\caption{The number of fixed points for ${T}_2$}
\label{fixedD}
\end{center}
\end{table}

The genus of $D$ is $9,$ so the character values of $\chi_{K_D}$ at the fourteen conjugacy classes ordered as above are as in Table \ref{charactertableD}.

\begin{table}[!ht]
\begin{center}
{\setlength\tabcolsep{2,1pt} \begin{tabular}{|c||c|c|c|c|c|c|c|c|c|c|c|c|c|c|} \hline
conjugacy class & $1$ & $g_5$ & $g_4$ & $g_4 g_5$ & $g_2 g_3 g_5$ & $g_2$ & $g_2 g_3$ & $g_3 g_4$ & $g_2 g_5$ & $g_3$  & $g_1$ & $g_1 g_2 g_3$ & $g_1 g_2$ & $g_1 g_3$ \\ \hline 
$ \chi_{K_D} $ & 9 & 1 & $-3$ & $-3$ & $-3$ & $-3$ & 1 & 1 & 1 & 1 & 1 & $-1$ & $-1$ & 1 \\\hline
\end{tabular} }
\caption{The values of $\chi_{K_D}$}
\label{charactertableD}
\end{center}
\end{table}

We find $\chi_{K_D} = \chi_4+ \chi_{10} + \chi_{12} + \chi_{13} + \chi_{14}$ from the character table again.

\subsection{Constructing line bundles on $C$}
Consider the normal subgroup $\langle g_5 \rangle \unlhd G(32,27)$ and the quotient map $\pi\colon C \longrightarrow B$ where $B$ is a quotient curve $B:=C/\langle g_5 \rangle$. Then by the unmixed ramification structure and the Riemann-Hurwitz formula, we see that $B$ is an elliptic curve and $\deg R=8$ where $R$ denotes a ramification divisor of $\pi$.
Moreover, since $G(32,27)$ act on $C,$ $\overline{G}:=G(32,27)/\langle g_5 \rangle$ also acts on $B$.
At first, we will show that there exists a non-trivial $\overline{G}$-invariant theta characteristic of $B$. To see this, note that there are 4 theta characteristics $\eta_0:=\mathcal{O}_B,$ $\eta_1,$ $\eta_2$ and $\eta_3$ on $B$.
Now consider the $\overline{G}$-orbit $\overline{G}\eta_i$ of $\eta_i$.
Since $\overline{G}$ acts on the set of all theta characteristics and $\eta_0=\mathcal{O}_B$ is fixed under the $\overline{G}$-action, we have 3 possibilities: ($\{i,j,k\}=\{1,2,3\}$)
\begin{enumerate}
	\item[(1)] $\overline{G}\eta_i=\{\eta_i\}$: This means $\eta_i$ is a $\overline{G}$-invariant, so we are done.	
	\item[(2)] $\overline{G}\eta_i=\{\eta_i, \eta_j\}$: 
			Take $\eta_k \notin \{\eta_i, \eta_j\}$. Then we have 
			$\overline{G}\eta_k=\{\eta_k\}$.
	\item[(3)] $\overline{G}\eta_i=\{\eta_i, \eta_j, \eta_k\}$: 
			By the orbit-stabilizer theorem, we have
			\[
			\lvert\overline{G}\eta_i\rvert
			\lvert\textrm{Stab}_{\overline{G}}(\eta_i)\rvert
			=\lvert\overline{G}\rvert
			\]
			where $\textrm{Stab}_{\overline{G}}(\eta_i)$ denotes the stabilizer of $\eta_i$ under the $\overline{G}$-action.
			However, since $3=\lvert\overline{G}\eta_i\rvert\nmid \lvert\overline{G}\rvert=16,$ this is impossible.
\end{enumerate}
Thus, there is a $\overline{G}$-invariant theta characteristic $\eta$ of $B$. We see that  $h^0(B,\eta)=0$ since there is no global section for any non-trivial line bundle of degree 0 of $B,$ and $h^1(B, \eta)$ also vanishes by the Riemann-Roch theorem. With the existence of such $\eta,$ we can prove the following lemma:

\begin{lem} 
$C$ has a $G(32,27)$-invariant ineffective theta characteristic $\mathcal{L}$.
\end{lem}
\begin{proof}
	Let $R$ be a ramification divisor of $\pi\colon C\longrightarrow B$ of degree 8 as above and consider a normal subgroup $H:=\langle g_2g_5, g_4 \rangle \unlhd G(32,27)$ of order 4.
	Then by the unmixed ramification structure and the Riemann-Hurwitz formula, $C/H\cong \mathbb{P}^1$ and since $g_2$ and $g_4$ freely act on $C,$ the $H$-action on $R$ is free. 
	Thus, $\mathcal{O}_C(R)=\pi_H^*\mathcal{O}_{\mathbb{P}^1}(1)^{\otimes 2}$ where $\pi_H\colon C\longrightarrow \mathbb{P}^1$ is the quotient map induced by $H$. Because $\mathcal{O}_{\mathbb{P}^1}(1)$ is $G(32,27)/H$-invariant, $\pi_H^*\mathcal{O}_{\mathbb{P}^1}(1)$ is $G(32,27)$-invariant. Therefore, $\mathcal{L}:=\pi^{*} \eta \otimes \pi_H^*\mathcal{O}_{\mathbb{P}^1}(1)$ is a $G(32,27)$-invariant line bundle and since
  	\begin{align*}
  	\mathcal{L}^{\otimes 2}
  	&= (\pi^{*} \eta \otimes \pi_H^*\mathcal{O}_{\mathbb{P}^1}(1))^{\otimes 2}
  	=\pi^{*}\eta^{\otimes2} \otimes \pi_H^*\mathcal{O}_{\mathbb{P}^1}(1)^{\otimes 2}
  	\\
  	&=\pi^{*}(K_B)\otimes\mathcal{O}_C(R)=K_C ,
  	\end{align*}
$\mathcal{L}$ is a theta characteristic. Moreover we have $h^0(C, \mathcal{L})=2h^0(B,\eta)=0$ by Proposition \ref{Beauville's lemma2}.
\end{proof}

\subsection{Constructing line bundles on $D$}
We begin with the following lemma which will be useful to construct some line bundle on $D$.
\begin{lem}\label{linequivofD}
Let $A_1, A_2, A_3, A_4$ be the set-theoretic fibers consisting of the ramification points of the map $\pi_D\colon D\to D/G(32,27)$ whose stabilizer group is $\langle g_2g_3g_4 \rangle$, $\langle g_2\rangle$, $\langle g_1g_2g_3g_5 \rangle$ and $\langle g_1g_2 \rangle$ respectively. Then the following linear equivalence relations hold:
	\begin{enumerate}
		\item[(1)] $A_1 \sim A_2 \sim 2A_3 \sim 2A_4$.
		\item[(2)] $A_3 \nsim A_4$.
	\end{enumerate}
\end{lem}
\begin{proof}
Note that the orders of the stabilizer groups $\langle g_2g_3g_4 \rangle$, $\langle g_2\rangle$, $\langle g_1g_2g_3g_5 \rangle$ and $\langle g_1g_2 \rangle$ are 2, 2, 4 and 4 respectively; see Remark \ref{G32isom}.
	\begin{enumerate}[(1)]
		\item Consider the subgroup $H_1=\langle g_1g_2g_4,g_4,g_5 \rangle$ of $G(32,27)$ of order 8 and the quotient curve $D/{H_1}$. 
Since $H_1$ is disjoint to conjugacy classes of stabilizer groups $\langle g_2g_3g_4 \rangle$, $\langle g_2\rangle$ of $A_1$ and $A_2$, respectively, $H_1$ acts freely on both $A_1$ and $A_2$. Similarly, since $(g_1g_2g_3g_5)^2=g_4g_5\in H_1$ and same holds for any other representatives of its conjugacy classes, $H_1$ acts on $A_3$ with stabilizer group $\mathbb{Z}_2$. Finally, since $g_1g_2$ and other representatives of its conjugacy classes are in $H_1$, $H_1$ acts on $A_4$ with stabilizer group $\mathbb{Z}_4$. Then, by the Riemann-Hurwitz formula, we have $D/{H_1} \cong \mathbb{P}^1$ and hence, we get $A_1 \sim A_2 \sim 2A_3$. Next, consider $H_2 = \langle g_1g_2g_3g_4g_5,g_4,g_5 \rangle$. By the similar argument for $H_2$, we can get $A_1 \sim A_2 \sim 2A_4$.
		\item Now consider the $H_3=\langle g_2, g_4 \rangle$-action on $D$. Then similar argument as above shows that $A_3$ contains no ramification point of $\pi_3\colon D\to D/H_3$ and $A_4$ contains some ramification points of $\pi_3$. Moreover, we can see that $A_3$ is given by a pull-back of some divisor of $D/H_3$. However, since $A_2$ also contains ramification points, $A_4$ does not contains all ramification points and hence, it cannot be a pull-back of some divisor of $D/H_3$ by Lemma \ref{Beauville's lemma1}. Therefore, $A_3 \nsim A_4$. \qedhere
	\end{enumerate}
\end{proof}


\begin{rmk}
From the above Lemma, we see that $\mathcal{O}(A_1-A_3)$, $\mathcal{O}(A_1-A_4)$, $\mathcal{O}(A_2-A_3)$, $\mathcal{O}(A_2-A_4)$ are effective line bundles.
\end{rmk}

From the above computations we also have the following Lemma.
\begin{lem}\label{h0ofM}
As a $G$-module, $H^0(D,\mathcal{O}(A_3)) \cong \chi_1 \oplus \alpha$ where $\alpha$ is a 2-dimensional irreducible representation of $G$.
\end{lem}
\begin{proof}
First, we claim that $\chi_1$ is the unique 1-dimensional subrepresentation of $H^0(D,\mathcal{O}(A_3))$. If there is another 1-dimensional subrepresentation of $H^0(D,\mathcal{O}(A_3)),$ then there should be a $G$-invariant effective divisor which is linearly equivalent to $A_3$. However $A_3$ is the unique $G$-invariant effective divisor of degree 8 linearly equivalent to itself. Therefore there is no 1-dimensional subrepresentation of $H^0(D,\mathcal{O}(A_3))$ other than $\chi_1$. 
Note that $D$ is not a hyperelliptic curve (cf. \cite{Pardini}). From the Clifford theorem we see that $1 \leq h^0(D,\mathcal{O}(A_3)) \leq 4$. Moreover from the analysis of the previous Lemma, we see that there is an effective divisor of degree 8 not equal to $A_3$ but linearly equivalent to $A_3$. Therefore $h^0(D,\mathcal{O}(A_3)) > 1$ and since there is no $n$-dimensional irreducible representation of $G$ for $n>2$(cf. Remark \ref{allcharactersarereal}), we get $H^0(D,\mathcal{O}(A_3))=\chi_1 \oplus \alpha$ where $\alpha$ is a 2-dimensional irreducible representation.
\end{proof}

Now we express $K_D$ in terms of $A_3$ and $A_4$.

\begin{lem}\label{KD=A3+A4}
$K_D \cong \mathcal{O}(A_3+A_4)$.
\end{lem}
\begin{proof}
Consider $H_4 = \langle g_4, g_5 \rangle$. From Riemann-Hurwitz formula we see that $D/{H_4}$ is an elliptic curve. Because $A_3+A_4$ are ramification divisor with stabilizer group $\mathbb{Z}_2$, we get $K_D \cong \mathcal{O}(A_3+A_4)$.
\end{proof}

Let $\mathcal{M} := \mathcal{O}(A_3)$ and $\mathcal{M}' := \mathcal{O}(A_4)$. Because $A_3$ and $A_4$ are $G$-invariant divisors of degree 8 in $D$, the action on the function field of $D$ induces natural $G$-linearizations on $\mathcal{M}$ and $\mathcal{M}'$. Therefore $\mathcal{M}$ and $\mathcal{M}'$ are $G$-equivariant lines bundles on $D$ of degree 8(see \cite{E1},\cite{E2}, \cite{Lee2} for more details). We also consider the map $H^0(D,\mathcal{O}(A_3)) \otimes H^0(D,\mathcal{O}(A_4)) \to H^0(D,\mathcal{O}(A_3+A_4))$. From Lemma \ref{h0ofM} and Lemma \ref{KD=A3+A4}, we see that $h^0(D,\mathcal{M})=h^0(D,\mathcal{M}')=3$. 

\begin{lem}\label{h1ofM}
As a G-module, $H^1(D,\mathcal{M}) \cong \chi_4 \oplus \beta$ where $\beta$ is an irreducible 2-dimensional representation of $G$.
\end{lem}
\begin{proof}
From Serre duality we see that $H^1(D,\mathcal{M}) \cong H^0(D,K_D \otimes \mathcal{M}^{-1})^*$. Consider the natural map $H^0(D,\mathcal{M}) \otimes H^0(D,K_D \otimes \mathcal{M}^{-1}) \to H^0(D,K_D)$. This map is the same as $H^0(D,\mathcal{O}(A_3)) \otimes H^0(D,\mathcal{O}(A_4)) \otimes \chi \to H^0(D,\mathcal{O}(A_3+A_4)) \otimes \chi$ for a character $\chi$. Because constant function belongs to $H^0(D,\mathcal{O}(A_3+A_4))$, $\chi_1$ is a $G$-submodule of $H^0(D,\mathcal{O}(A_3+A_4))$. From Lefschetz fixed point formula we see that $H^0(D,K_D) = \chi_4 + \chi_{10} + \chi_{12} + \chi_{13} + \chi_{14}$. Note that $\chi_{10}, \chi_{12}, \chi_{13}, \chi_{14}$ are 2-dimensional irreducible representations of $G$. Therefore we get $\chi=\chi_4$ and $H^1(D,\mathcal{M}) \cong H^0(D,K_D \otimes \mathcal{M}^{-1})^*=\chi_4 \oplus \beta$ where $\beta$ is a 2-dimensional irreducible representation.
\end{proof}

\subsection{Exceptional sequences of line bundles on $S$}

It is well known that $\textsf{D}^b_G(C \times D) \simeq \textsf{D}^b(S)$ since $G$ acts freely on $C \times D$(see \cite[Appendix]{Vistoli} for more details). Therefore it suffices to construct exceptional sequence in $\textsf{D}^b_G(C \times D)$. We consider equivariant line bundles on $C \times D$. By abuse of notation, we denote a $G$-equivariant line bundle on $C \times D$ and its descent on $S$ by the same symbol. 
Note that $\mathcal{L}$ is not a $G$-equivariant line bundle on $C$. However from the Dolgachev's theorem we can prove that there exists a $G$-invariant line bundle $\mathcal{N}$ such that $\mathcal{L} \boxtimes \mathcal{N}$ is a $G$-equivariant line bundle on $C \times D$. 

\begin{theorem}\label{maintheorem} There is a character $\chi \in \widehat{G}$ such that $\mathcal{L} \boxtimes (\mathcal{M} \otimes \mathcal{N})(\chi),$ $\mathcal{L} \boxtimes \mathcal{N},$ $\mathcal{O}_C \boxtimes \mathcal{M}(\chi),$ $\mathcal{O}_C \boxtimes \mathcal{O}_D$ descend to an exceptional sequence of line bundles on $S$.
\end{theorem}
\begin{proof}
Because $S$ is a surface with $p_g=q=0$, every line bundle on $S$ is an exceptional object. Now from the K\"{u}nneth formula, we see that, for all $i$,
	\[
		H^i(S,\mathcal{L} \boxtimes \mathcal{N})
		=
		\left(\bigoplus_{j+k=i} H^j(C,\mathcal{L}) 
		\otimes H^k(D,\mathcal{N})\right)^G=0,
	\]
	\[
		H^i(S,\mathcal{L} \boxtimes (\mathcal{M} \otimes \mathcal{N})(\chi))
		=
		\left(\bigoplus_{j+k=i} H^j(C,\mathcal{L}) 
		\otimes H^k(D,(\mathcal{M} \otimes \mathcal{N})) \otimes \chi\right)^G=0,
	\]
	\[
		H^i(S,\mathcal{L} \boxtimes (\mathcal{M}^{-1} \otimes \mathcal{N})(\chi^{-1}))
		=
		\left(\bigoplus_{j+k=i} H^j(C,\mathcal{L}) \otimes 
		H^k(D,(\mathcal{M}^{-1} \otimes \mathcal{N})) \otimes \chi^{-1}\right)^G=0
	\]
as $\mathcal{L}$ is an ineffective theta characteristic on $C$.
From Riemann-Roch formula, the Euler-Poincar\'{e} characteristic of $\mathcal{O}_C \boxtimes \mathcal{M}(\chi)$ on $S$ is equal to 0.
Then by K\"{u}nneth formula we see that it is enough to show that the $G$-invariant parts of the following vector spaces are all zero.
	\begin{align*}
		&H^0(S,\mathcal{O}_C \boxtimes \mathcal{M}(\chi))
		=(H^0(C,\mathcal{O}_C) \otimes H^0(D,\mathcal{M}) \otimes \chi)^G\\
		&H^2(S,\mathcal{O}_C \boxtimes \mathcal{M}(\chi))
		=(H^1(C,\mathcal{O}_C) \otimes H^1(D,\mathcal{M}) \otimes \chi)^G
	\end{align*}
From the equivariant Serre duality we see that 
$ H^1(C,\mathcal{O}_C) \cong H^0(C,K_C)^* $
as $G$-modules and we know the representation of $H^1(C,\mathcal{O}_C)=\chi_7+\chi_9+\chi_{11}$.
Moreover $H^0(D,\mathcal{M})=\chi_1 \oplus A$ and $H^1(D,\mathcal{M})=\chi_4 \oplus B$ where $A$ and $B$ are 2-dimensional irreducible representations of $G$ (cf. Lemma \ref{h0ofM}, Lemma \ref{h1ofM}).
Finally we can check that for any possible representation $H^0(D,\mathcal{M})$, $H^1(D,\mathcal{M})$ there exists $\chi \in \widehat{G}$ such that 
	\[
		h^0(S,\mathcal{O}_C \boxtimes \mathcal{M}(\chi))
		=h^1(S,\mathcal{O}_C \boxtimes \mathcal{M}(\chi))
		=h^2(S,\mathcal{O}_C \boxtimes \mathcal{M}(\chi))=0
	\]
using Magma (cf. \ref{appendix B}).
In other word, we can always find $\chi \in \widehat{G}$ such that a sequence of 4 equivariant line bundles $\mathcal{L} \boxtimes (\mathcal{M} \otimes \mathcal{N})(\chi),$ $\mathcal{L} \boxtimes \mathcal{N},$ $\mathcal{O}_C \boxtimes \mathcal{M}(\chi),$ $\mathcal{O}_C \boxtimes \mathcal{O}_D$ on $C \times D$ descent to an exceptional sequence of line bundles on $S$.
\end{proof}

\subsection{Quasiphantom categories}
From the exceptional collections of maximal length 4, we obtain examples of quasiphantom categories. We recall the definitions of quasiphantom and phantom category.

\begin{defi}{\cite[Definition 1.8]{GO}}
Let $S$ be a smooth projective variety. Let $\mathcal{A}$ be an admissible triangulated subcategory of $\textsf{D}^b(S)$. Then $\mathcal{A}$ is called a quasiphantom category if the Hochschild homology of $\mathcal{A}$ vanishes, and the Grothendieck group of $\mathcal{A}$ is finite. If the Grothendieck group of $\mathcal{A}$ also vanishes, then $\mathcal{A}$ is called a phantom category.
\end{defi}

Now we obtain new examples of quasiphantom categories as follows.

\begin{prop}
The orthogonal complements of the exceptional sequence of Theorem \ref{maintheorem} are quasiphantom categories.
\end{prop}
\begin{proof}
There is a semiorthogonal decomposition $\textsf{D}^b(S)=\langle \mathcal{A}, \mathcal{B} \rangle$ where $\mathcal{B}$ is the full triangulated subcategory of $\textsf{D}^b(S)$ generated by exceptional collection of length 4 constructed above. From \cite{Kuz}, we see that the Hochschild homology of $ \mathcal{A}$ vanishes and from \cite{BCF}, we see that the Grothendieck group of $ \mathcal{A}$ is $\mathbb{Z}_2^2 \oplus \mathbb{Z}_4 \oplus \mathbb{Z}_8$. Therefore $\mathcal{A}$ is a quasiphantom category.
\end{proof}

\appendix

\section{$G(32,27)$}\label{appendix A}
For the convenience of readers, we list basic properties of $G(32,27)$.
The following data were obtained using GAP \cite{GAP} and Magma \cite{Magma}. $SmallGroup(a,b)$ denotes a finite group of order $a$ having number $b$ in the list in Magma library.

\renewcommand*{\appendixname}{}

\begin{defi}
 We write $G(32,27)$ using finite polycyclic presentation.
	\begin{align*}
		G(32,27)
		&:=SmallGroup(32,27)\\
		&=\langle g_1,g_2,g_3,g_4,g_5\mid g_1^{-1}g_2g_1=g_2g_4,g_1^{-1}g_3g_1=g_3g_5 \rangle.
	\end{align*}
\end{defi}

\begin{rmk}\label{G32isom}
We see that $G(32,27)$ is a semidirect product of $N=\mathbb{Z}_2^4$ and $Q=\mathbb{Z}_2$ via the following isomorphism (cf. \cite{BCG2}). 
	\[
		\begin{array}{ccccc}
		g_1 \mapsto (0,1), && g_2 \mapsto ((1,0,0,0),0), &&g_3 \mapsto ((0,1,0,0),0),\\
		g_4 \mapsto ((0,0,1,0),0), && g_5 \mapsto ((0,0,0,1),0). &&
		\end{array}
	\]
Here we write $N \rtimes_{\Phi} Q$ to denote the semidirect product where $\Phi \colon Q \to \textup{Aut}(N) \cong GL_4(\mathbb{F}_2)$ can be represented by the following matrix.
$$ \Phi_1=\begin{pmatrix}
1 & 0 & 0 & 0 \\
0 & 1 & 0 & 0  \\
1 & 0 & 1 & 0 \\
0 & 1 & 0 & 1
\end{pmatrix} $$
The multiplication of $G(32,27)=N \rtimes_{\Phi} Q$ can be defined as follows:
$$ (n_1,q_1) \cdot (n_2,q_2)=(n_1+\Phi_{q_1}(n_2),q_1+q_2). $$
\end{rmk}

The list of conjugacy classes of $G(32,27)$ is as follows.
\[
[ 1 ], [ g_5 ], [ g_4 ], [ g_4g_5 ], [ g_2g_3g_4, g_2g_3g_5 ],  [ g_2, g_2g_4 ], [ g_2g_3, g_2g_3g_4g_5 ], 
\]
\[
 [ g_3g_4, g_3g_4g_5 ], [ g_2g_5, g_2g_4g_5 ], [ g_3, g_3g_5 ], [ g_1, g_1g_4, g_1g_5, g_1g_4g_5 ],
\]
\[
[ g_1g_2g_3, g_1g_2g_3g_4, g_1g_2g_3g_5, g_1g_2g_3g_4g_5 ],
\]
\[
[ g_1g_2, g_1g_2g_4, g_1g_2g_5, g_1g_2g_4g_5 ],  [ g_1g_3, g_1g_3g_4, g_1g_3g_5, g_1g_3g_4g_5 ].
\]

The list of normal subgroups of $G(32,27)$ is as follows.
	\begin{eqnarray*}
	& G(32,27), \langle g_1, g_3, g_4, g_5 \rangle, 
	\langle g_1, g_2g_3, g_4, g_5 \rangle,  \langle g_1, g_2, g_4, g_5 \rangle,&\\ 
	& \langle g_2, g_3, g_4, g_5 \rangle, \langle g_1g_2g_4, g_3, g_4, g_5 \rangle, 
	\langle g_1g_3g_5, g_2, g_4, g_5 \rangle,&\\
	&\langle g_1g_2g_4, g_2g_3g_4g_5, g_4, g_5 \rangle, \langle g_3, g_4, g_5 \rangle,
	\langle g_1g_3g_5, g_4, g_5 \rangle, \langle g_2g_3, g_4, g_5 \rangle,&\\
	&\langle g_1g_2g_3g_4g_5, g_4, g_5 \rangle, \langle g_2, g_4, g_5 \rangle, 
	\langle g_1g_2g_4, g_4, g_5 \rangle, \langle g_1, g_4, g_5 \rangle,&\\
	&\langle g_2, g_4 \rangle, \langle g_2g_5, g_4 \rangle, \langle g_4, g_5 \rangle, 
	\langle g_3, g_5 \rangle, \langle g_3g_4, g_5 \rangle,\\
	&\langle g_2g_3, g_4g_5 \rangle, \langle g_2g_3g_4, g_4g_5 \rangle, 
	\langle g_5 \rangle, \langle g_4g_5 \rangle, \langle g_4 \rangle, \langle 1 \rangle&
	\end{eqnarray*}

Table \ref{charactertableG} is the character table of $G(32,27)$. \\

\begin{table}[!th]
\begin{center}
{\setlength\tabcolsep{2,2pt}  
\begin{tabular}{|c||c|c|c|c|c|c|c|c|c|c|c|c|c|c|} \hline
 & $1$ & $g_5$ & $g_4$ & $g_4g_5$ & $g_2g_3g_5$ & $g_2$ & $g_2g_3$ & $g_3g_4$ & $g_2g_5$ & $g_3$ & $g_1$ & $g_1g_2g_3$ & $g_1g_2$ & $g_1g_3$  \tabularnewline
\hline\hline
$\chi_{1}$ & 1 & 1 & 1 & 1 & 1 & 1 & 1 & 1 & 1 & 1 & 1 & 1 & 1 & 1 \tabularnewline
\hline
$\chi_{2}$ & 1 & 1 & 1 & 1 & $-1$ & 1 & $-1$ & $-1$ & 1 & $-1$ & 1 & $-1$ & 1 & $-1$ \tabularnewline
\hline
$\chi_{3}$ & 1 & 1 & 1 & 1 & 1 & $-1$ & 1 & $-1$ & $-1$ & $-1$ & 1 & 1 & $-1$ & $-1$ \tabularnewline
\hline
$\chi_{4}$ & 1 & 1 & 1 & 1 & $-1$ & $-1$ & $-1$ & 1 & $-1$ & 1 & 1 & $-1$ & $-1$ & 1 \tabularnewline
\hline
$\chi_{5}$ & 1 & 1 & 1 & 1 & $-1$ & 1 & $-1$ & $-1$ & 1 & $-1$ & $-1$ & 1 & $-1$ & 1 \tabularnewline
\hline
$\chi_{6}$ & 1 & 1 & 1 & 1 & 1 & 1 & 1 & 1 & 1 & 1 & $-1$ & $-1$ & $-1$ & $-1$ \tabularnewline
\hline
$\chi_{7}$ & 1 & 1 & 1 & 1 & $-1$ & $-1$ & $-1$ & 1 & $-1$ & 1 & $-1$ & 1 & 1 & $-1$ \tabularnewline
\hline
$\chi_{8}$ & 1 & 1 & 1 & 1 & 1 & $-1$ & 1 & $-1$ & $-1$ & $-1$ & $-1$ & $-1$ & 1 & 1 \tabularnewline
\hline
$\chi_{9}$ & 2 & $-2$ & 2 & $-2$ & 0 & 2 & 0 & 0 & $-2$ & 0 & 0 & 0 & 0 & 0 \tabularnewline
\hline
$\chi_{10}$ & 2 & 2 & $-2$ & $-2$ & 0 & 0 & 0 & $-2$ & 0 & 2 & 0 & 0 & 0 & 0 \tabularnewline
\hline
$\chi_{11}$ & 2 & $-2$ & $-2$ & 2 & 2 & 0 & $-2$ & 0 & 0 & 0 & 0 & 0 & 0 & 0 \tabularnewline
\hline
$\chi_{12}$ & 2 & 2 & $-2$ & $-2$ & 0 & 0 & 0 & 2 & 0 & $-2$ & 0 & 0 & 0 & 0 \tabularnewline
\hline
$\chi_{13}$ & 2 & $-2$ & 2 & $-2$ & 0 & $-2$ & 0 & 0 & 2 & 0 & 0 & 0 & 0 & 0 \tabularnewline
\hline
$\chi_{14}$ & 2 & $-2$ & $-2$ & 2 & $-2$ & 0 & 2 & 0 & 0 & 0 & 0 & 0 & 0 & 0 \tabularnewline
\hline
\end{tabular}
}
\caption{Character table of $G(32,27)$}
\label{charactertableG}
\end{center}
\end{table}

\begin{rmk}\label{allcharactersarereal}
Note that all characters of $G(32,27)$ are real valued and there is no $n$-dimensional irreducible representation for $n>2$.
\end{rmk}

\renewcommand*{\appendixname}{Appendix }

\section{Magma code}\label{appendix B}

\renewcommand*{\appendixname}{}

A downloadable code which readers can just copy and paste into Magma is uploaded at \url{https://sites.google.com/site/kklmagmacodes/} \quad Those who do not have a Magma license can use the online Magma calculator at 
\url{http://magma.maths.usyd.edu.au/calc/}

The following is the basic setup, to ensure that the conjugacy classes and the character table are arranged in the way we presented in our paper. The commands following the symbol \url{>} are to be entered into Magma.
\begin{verbatim}
> G := SmallGroup(32,27);
> G;
GrpPC : G of order 32 = 2^5
PC-Relations:
    G.2^G.1 = G.2 * G.4,
    G.3^G.1 = G.3 * G.5
> cls := Classes(G);
> cls_map := ClassMap(G);
> ct := CharacterTable(G);
> R := ClassFunctionSpace(G);
> cls_reps := [ Identity(G), G.5, G.4, G.4*G.5, G.2*G.3*G.5, G.2, G.2*G.3,
>             G.3*G.4, G.2*G.5, G.3, G.1, G.1*G.2*G.3, G.1*G.2, G.1*G.3 ];
> cls_reps_ind := [ cls_map(cls_reps[i]) : i in [1..#cls] ];
> function Character_for_our_cls_reps(values_at_cls_reps)
>        return R![ values_at_cls_reps[Index([1..#cls], cls_reps_ind[i])] : 
>                   i in [1..#cls] ];
> end function;
\end{verbatim}
For example, we check whether the first and the fourth irreducible characters of \url{ct} coincide respectively with $\chi_1$ and $\chi_4$ in the table we presented in Appendix A; in principle, one could do such a check for every $\chi_i$ in the table in Appendix A.
\begin{verbatim}
> Character_for_our_cls_reps([1,1,1,1,-1,-1,-1,1,-1,1,1,-1,-1,1]) eq ct[4];
true
> Character_for_our_cls_reps([1,1,1,1,1,1,1,1,1,1,1,1,1,1]) eq ct[1];
true
\end{verbatim}

The following code is to get the results of \S3.1. We first compute the number of fixed points of elements of $G(32,27)$ for the ramification structure $T_1$, and will compute $\chi_{K_C}$ accordingly.
\begin{verbatim}
> H_1 := sub<G|G.4*G.5*G.1>;
> H_2 := sub<G|G.2*G.3*G.4*G.5>;
> H_3 := sub<G|G.2*G.4*G.5>;
> H_4 := sub<G|G.3*G.4*G.5*G.1>;
> standard_stabs_T1 := [H_1,H_2,H_3,H_4];
> subs_T1 := &cat[[H^tr : tr in Transversal(G,H)] : H in standard_stabs_T1];
\end{verbatim}
For $i=1,2,3,4$, $H_i=$\url{H_i} is the stabilizer group of one of the pre-images of the $i$-th ramification point; call this pre-image $x_i$. Let's assume that $G(32,27)$ is acting from right. If it was acting from left, then we can turn it into a right action by letting each element of $G(32,27)$ act by its inverse; when doing so, the stabilizer groups $H_i$ do not change. Let $g_{i1}, g_{i2}, \cdots$ be the right coset representatives of $H_i$ in $G(32,27)=$\url{G}; so $H_i g_{i_1}$, $H_i g_{i2}$, $\cdots$ are the cosets. Then the list of all distinct pre-images of the $i$-th ramification point is $x_i \, g_{i1}$, $x_i \, g_{i2}$, $\cdots$, whose stabilizers are $g_{i1}^{-1} H_i g_{i1}$, $g_{i2}^{-1} H_i g_{i2}$, $\cdots$. So \url{subs_T1} is the collection of all these stabilizer groups for all $i=1,2,3,4$, counted with multiplicity; the number of fixed points of an element $g$ is the number of all stabilizer groups in \url{subs_T1} that contains $g$.
\begin{verbatim}
> fixedpts_T1 := [ #[sub : sub in subs_T1 | cls_reps[i] in sub ] 
>                 : i in [2..#cls] ];
> fixedpts_T1;
[ 8, 0, 0, 0, 0, 8, 0, 8, 0, 4, 0, 0, 4 ]
\end{verbatim}
Then the character value for $\chi_{K_C}$ and its decomposition into irreducible characters are obtained as follows.
\begin{verbatim}
> values_C := [ (2-fixedpts_T1[i])/2 : i in [1..#fixedpts_T1]];
> values_C;
[ -3, 1, 1, 1, 1, -3, 1, -3, 1, -1, 1, 1, -1 ]
> chi_K_C_at_cls_reps := [ 5 ] cat values_C;
> chi_K_C := Character_for_our_cls_reps(chi_K_C_at_cls_reps);
> [InnerProduct(chi_K_C, ct[i]) : i in [1..#ct]];
[ 0, 0, 0, 0, 0, 0, 1, 0, 1, 0, 1, 0, 0, 0 ]
> chi_K_C eq ct[7]+ct[9]+ct[11];
true
\end{verbatim}
We do likewise for $T_2$, to get $\chi_{K_D}$:
\begin{verbatim}
> H_5 := sub<G|G.2*G.3*G.4>;
> H_6 := sub<G|G.2>;
> H_7 := sub<G|G.2*G.3*G.4*G.1>;
> H_8 := sub<G|G.2*G.4*G.1>;
> standard_stabs_T2 := [H_5,H_6,H_7,H_8];
> subs_T2 := &cat[[H^tr : tr in Transversal(G,H)] : H in standard_stabs_T2];
> fixedpts_T2:=[#[sub:sub in subs_T2|cls_reps[i] in sub]:i in [2..#cls]];
> fixedpts_T2;
[ 0, 8, 8, 8, 8, 0, 0, 0, 0, 0, 4, 4, 0 ]
> values_D := [ (2-fixedpts_T2[i])/2 : i in [1..#fixedpts_T2]];
> values_D;
[ 1, -3, -3, -3, -3, 1, 1, 1, 1, 1, -1, -1, 1 ]
> chi_K_D_at_cls_reps := [ 9 ] cat values_D;
> chi_K_D := Character_for_our_cls_reps(chi_K_D_at_cls_reps);
> [InnerProduct(chi_K_D, ct[i]) : i in [1..#ct]];
[ 0, 0, 0, 1, 0, 0, 0, 0, 0, 1, 0, 1, 1, 1 ]
> chi_K_D eq ct[4]+ct[10]+ct[12]+ct[13]+ct[14];
true
\end{verbatim}

The following Magma code enables us to check that for each 2-dimensional irreducible representations $A,B,$ there exists $\chi \in \widehat{G}$ such that 
$$ H^0(S,\mathcal{O}_C \boxtimes \mathcal{M}(\chi))^G=(H^0(C,\mathcal{O}_C) \otimes H^0(D,\mathcal{M}) \otimes \chi)^G = 0 $$
$$ H^2(S,\mathcal{O}_C \boxtimes \mathcal{M}(\chi))^G=(H^1(C,\mathcal{O}_C) \otimes H^1(D,\mathcal{M}) \otimes \chi)^G = 0 $$
hold where $H^0(C,\mathcal{O}_C) \cong \chi_1$ and $H^1(C,\mathcal{O}_C)=  \chi_7 + \chi_9 + \chi_{11}$ in the character table, $H^0(D,\mathcal{M})=\chi_1 \oplus A,$ $H^1(D,\mathcal{M})=\chi_4 \oplus B$ (cf. proof of Theorem \ref{maintheorem}).
\begin{verbatim}
> H0_C_OC := ct[1];
> H1_C_OC := chi_K_C;
> for i in [9..14] do               // irreducible characters of dimension 2
> for j in [9..14] do               // irreducible characters of dimension 2
>     H0_D_M := ct[1] + ct[i];      // A is ct[i]
>     H1_D_M := ct[4] + ct[j];      // B is ct[j]
>     good_chis := [k:k in [1..8]|  // irreducible characters of dimension 1
>                 InnerProduct(H0_C_OC * H0_D_M * ct[k], ct[1]) eq 0 and 
>                 InnerProduct(H1_C_OC * H1_D_M * ct[k], ct[1]) eq 0];
>     if #good_chis eq 0 then
>            printf "\n (i,j)=(%o,%o) is trouble.", i,j;
>     // else
>     //     printf "\n A=ct[%o], B=ct[%o] -> chi=ct%o works", i,j,good_chis;
>     end if;
> end for;
> end for;
\end{verbatim}
We believe that anyone could easily interpret what the above code is doing. If the code does not print any message, it means that the statement is true. Removing ${\rm //}$ from the two lines lets the reader see what $\chi$'s work for each $A$ and $B$. 

\end{document}